\documentclass[12pt]{amsart}
\usepackage{amsmath, amsthm, amscd, amsfonts, amssymb, mathrsfs, graphicx, color}
\usepackage[bookmarksnumbered, colorlinks, plainpages]{hyperref}

\newtheorem{theorem}{Theorem}[section]

\newtheorem{corollary}[theorem]{Corollary}

\newtheorem{lemma}[theorem]{Lemma}
\theoremstyle{definition}

\newtheorem{notation}[theorem]{Notation}

\begin{document}
	\title[The influence of the prime omega function]{The influence of the prime omega function on the product of element orders in finite groups}
\author[M. Baniasad Azad \& M. Arabtash]{Morteza Baniasad Azad \& Mostafa Arabtash}
\address{Abolfazl Street, 22 Bahman Blvd., Bardsir, Kerman 78416-64979, Iran \newline }
\email{baniasad84@gmail.com}
\address{School of Information and Engineering, Dalarna University, Sweden\newline }
\email{mob@du.se}

\thanks{}

	
	\begin{abstract}
Let $G$ be a finite group and define $\rho(G) = \prod_{x \in G} o(x)$, where $o(x)$ denotes the order of the element $x \in G$.  Let $\Omega$ be the prime omega function giving the number of (not necessarily distinct) prime factors of a natural number.

In this paper,  we consider the function $\Omega_{\rho}(G):= \Omega(\rho(G))$. We show that, under certain conditions, this function exhibits behavior analogous to the derivative in calculus. We establish the following results:
\textbf{(Product rule)}	If $A$ and $B$ are finite groups, where $\operatorname{gcd}(|A|,|B|)=1$, then $\Omega_{\rho}(A\times B) =  \Omega_{\rho}(A) \cdot |B|+\Omega_{\rho}(B) \cdot |A|$. \\
\textbf{(Quotient rule)}
If $P$ is a  central cyclic normal Sylow $p$-subgroup of a finite group $G$, then
$	\Omega_{\rho}(\dfrac{G}{P}) = \dfrac{\Omega_{\rho}(G)\cdot|P|-\Omega_{\rho}(P)\cdot |G|}{{|P|}^2}.$ \\
Moreover, we show that
if $C$ is a cyclic group and $G$ is a non-cyclic group of the
same order, then $\Omega_{\rho}(G) \leq \Omega_{\rho}(C)$.
Finally, we show that if $G$ is a group of order $|L_2(p)|$,
then $\Omega_{\rho}(G) \geqslant \Omega_{\rho}(L_2(p))$, where $p \in \{5, 11, 13\} $.
	\end{abstract}

	\maketitle
	\section{\bf Introduction and Preliminary Results}
All groups considered in this paper are assumed to be finite.  
For a finite group $G$, let $\rho(G)$ denote the product of the orders of its elements, i.e., $\rho(G) = \prod_{x \in G} o(x)$.  
For any positive integer $n$ and any non-cyclic group $G$ of order $n$, if $C_n$ denotes the cyclic group of order $n$, then the inequality $\rho(G) < \rho(C_n)$ holds (see \cite[Theorem 3]{PatassinizbMATH06708963}).

Moreover, for any non-cyclic supersolvable group of order $n$, if either $G$ is nilpotent or $G$ is not metacyclic, then it has been shown that
\[
\rho(G) \leq q^{-n(q-1)/q} \rho(C_n),
\]
where $q$ is the smallest prime divisor of $n$ (see \cite[Theorem A]{noce}).

	Let $l(G)=\sqrt[|G|]{\rho(G)}/|G|$.
	In \cite[Theorem B]{grazian2024structure}, it has been established that if $l(G) > l(D_{2p})$, where $p$ is an odd prime dividing the order of $G$, then the group $G$ is $p$-nilpotent. In \cite{baniAust}, it was shown that if  $l(G) > l(S_3)$, $l(G) > l(A_4)$, or $l(G) > l(A_5)$, then $G$ must be nilpotent, supersolvable, or solvable, respectively.
	
	Furthermore, in \cite{baniMal}, it was demonstrated that the simple groups $L_2(7)$ and $L_2(11)$ are uniquely determined by the product of their element orders. In a series of subsequent studies \cite{banihamedapplication,banihamedcom}, several families of groups were identified as uniquely characterized by this product. Among such families are the groups of the form $A_5 \times C_p$, where $p > 5$ is a prime number.

In number theory, the prime omega function $\Omega(n)$ counts the total number of prime factors of a positive integer $n$, with multiplicity. Formally, if $n$ is a positive integer with the prime factorization 
$n = p_1^{\alpha_1} p_2^{\alpha_2} \cdots p_k^{\alpha_k}$,
where $p_1, p_2, \dots, p_k$ are distinct prime numbers, then
$\Omega(n) = \alpha_1 + \alpha_2 + \cdots + \alpha_k$.
 The function $\Omega(n)$ is completely additive.
For a finite group $G$, we define  $ \Omega_{\rho}(G):= \Omega(\rho(G)) $.
As an example, if $G \cong A_5$, then 
$$\rho(G) = 2^{15}\cdot3^{20}\cdot5^{24},  \qquad \Omega_{\rho}(G)=59. $$

In this paper, first we prove that:
\begin{itemize}
\item
\textbf{(Product Rule).}
If $A$ and $B$ are finite groups, then the following inequality holds:
\[
\Omega_{\rho}(A \times B) \leq \Omega_{\rho}(A) \cdot |B| + \Omega_{\rho}(B) \cdot |A|.
\]
Moreover, equality holds, that is,
\begin{align*}
	\Omega_{\rho}(A \times B) = \Omega_{\rho}(A) \cdot |B| + \Omega_{\rho}(B) \cdot |A|,
\end{align*}
if and only if $\gcd(|A|, |B|) = 1$.

\item
\textbf{(Quotient Rule).}
Let $P$ be a cyclic normal Sylow $p$-subgroup of a finite group $G$. Then it is shown that
\[
\Omega_{\rho}\left( \frac{G}{P} \right) \geq \frac{\Omega_{\rho}(G) \cdot |P| - \Omega_{\rho}(P) \cdot |G|}{|P|^2}.
\]
Furthermore, equality holds,
\begin{align*}
	\Omega_{\rho}\left( \frac{G}{P} \right) = \frac{\Omega_{\rho}(G) \cdot |P| - \Omega_{\rho}(P) \cdot |G|}{|P|^2},
\end{align*}
if and only if $P$ is central in $G$.
\end{itemize}

Moreover, we show that
if $C$ is a cyclic group and $G$ is a non-cyclic group of the
same order, then $\Omega_{\rho}(G) \leq \Omega_{\rho}(C)$.
Finally, we show that if $G$ is a group of order $|L_2(p)|$,
then $\Omega_{\rho}(G) \geqslant \Omega_{\rho}(L_2(p))$, where $p \in \{5, 11, 13\} $.

To establish these results, several auxiliary lemmas will be required.
\begin{lemma}\cite[Corollary 2.10]{amiri} \label{amiri}
	Let $ G $ be a finite  group of order $ n $. Then there exists a bijection $ f $ from $ G $ onto $ C_n $
	such that $ o(x) | o(f(x)) $ for all $ x \in G $.
\end{lemma}
	\begin{lemma}\cite[Theorem 3]{PatassinizbMATH06708963} \label{cyclic}
	Let $G$ be a finite group of order $n$. Then $\rho(G) \leq \rho(C_n)$
	with equality if and only if $G$ is cyclic.
\end{lemma}
	\begin{lemma} \label{mid1}   \cite[Lemma 2.6]{baniAust}
	Let $ G  $ be a finite group satisfying  $ G = P  \rtimes F$, where $P$ is a cyclic $p$-group
	for some prime $p$, $|F| > 1$ and $(p,|F|) = 1$. Then
	\[\rho(G) = \rho(P)^{|C_{F}(P)|} \rho(F)^{|P|}.\]
\end{lemma}
\begin{lemma}\label{1sylowproductt}	\cite[Lemma 2.4]{baniAust}
	Let $P$ be a cyclic normal Sylow $p$-subgroup of a finite group $G$. Then
	\[\rho(G) \mid	\rho(G/P)^{|P|}\rho(P)^{|G/P|}\]
	with equality if and only if $P$ is central in $G$.
\end{lemma}
	\begin{lemma}\cite{hallzbMATH02575957} \label{hall} 
	An integer $n = {p_1}^{\alpha_1} \dots {p_k}^{\alpha_k}$
	is the number of Sylow $p$-subgroups of a
	finite solvable group $G$ if and only if  ${p_i}^{\alpha_i} \equiv 1 \pmod{p}$
	for $i = 1, \dots, k$.
\end{lemma}

\begin{lemma}\cite{huppert} \label{huppert}
	Let $ G = L_2(q) $ where $ q $ is a ppower ($ p $ prime). Then\\
	(1) a Sylow $ p $-subgroup $ P $ of $ G $ is an elementary abelian group of
	order $ q $ and the number of Sylow $ p $-subgroup of $ G $ is $ q + 1 $,\\
	(2) $ G $ contains a cyclic subgroup $ A $ of order $\dfrac{q-1}{2}  $ such that $ N_G(u) $
	is a dihedral group of order $ q - 1 $ for every nontrivial element
	$ u \in A $,\\
	(3) $ G $ contains a cyclic subgroup $ B $ of order $ \dfrac{q+1}{2} $ such that $ N_G(u) $
	is a dihedral group of order $ q + 1 $ for every nontrivial element
	$ u \in B $,\\
	(4) the set $ \{ P^x, A^x, B^x | x \in G \} $ is a partition of $ G $.
\end{lemma}
\begin{lemma}\cite[Lemma 1]{Xu} \label{seri}
 Let $G$ be a non-solvable group. Then $G$ has a normal series
$1 \unlhd H  \unlhd K \unlhd G$ such that $K/H$ is a direct product of isomorphic non-abelian simple groups
and $|G/K|$ divides  $|Out(K/H)|$.
\end{lemma}

	\section{\bf On the function  $ \Omega_{\rho}(G) $}
	\begin{notation}
		Let  $n={p_1}^{\alpha_1}  {p_2}^{\alpha_2} \cdots  {p_k}^{\alpha_k}$,
		where $p_1, p_2, \cdots, p_k$
		are distinct prime numbers.
		We set $[n]_{p_i}=\alpha_i$, where $i=1, 2, \cdots, k$.
	\end{notation}
	
		\begin{lemma} \cite[Lemma 2.2 and Lemma 2.4]{baniMalArab} \label{mohem}  
		Let $G$ be a finite group of order $kp^{\alpha}$ such that  $(k,p)=1$ and $p$ is a  prime  number. Then
		 \[
		k \mid [\rho(G)]_p \quad \text{and} \quad \phi(p) \mid [\rho(G)]_p.
		\]
	\end{lemma}

	\begin{lemma}\cite[Corollary 2.6]{baniMalArab} \label{2-part}
	For a finite group $G$,  $[\rho(G)]_p$ is even, where $p$ is an odd prime number and $p \mid |G|$. 
	Moreover, if $2 $ divides $|G|$, then $[\rho(G)]_2$ is odd.
\end{lemma}
	Obviously we see that:
	\begin{lemma}
		Let $G$ be a finite group.
		\begin{itemize}
			\item If $H \leqslant G$, then $\Omega_{\rho}(H) \leq \Omega_{\rho}(G)$, 	with equality if and only if $H = G$. 
			\item  If $N \unlhd G$, then $\Omega_{\rho}(G/N) \leq \Omega_{\rho}(G)$, 	with equality if and only if $N = 1$.
		\end{itemize}
	\end{lemma}
	\begin{lemma} We have:
		\begin{enumerate}
			\item[(i)] $\Omega_{\rho}\left(C_{p^\alpha}\right)=\alpha p^{\alpha}-\dfrac{p^{\alpha}-1}{p-1}={\frac{\alpha p^{\alpha+1}-(\alpha+1) p^\alpha+1}{p-1}}$
			\item[(ii)] 
			$\Omega_{\rho}\left(C_{p^\alpha} \times C_{p^\beta}\right)={\frac{\beta p^{\alpha+\beta+2}-p^{\alpha+\beta+1}-(\beta+1) p^{\alpha+\beta}+p^{2 \alpha+1}+1}{p^2-1}}$
			\item[(iii)]
			$\Omega_{\rho}(L_2(q))={\frac{q(q+1)}{2}}{\Omega_{\rho}(C_{\frac{q-1}{2}})}+ {\frac{q(q-1)}{2}} {\Omega_{\rho}(C_{\frac{q+1}{2}})}+ {(q+1)(q-1)}.$
		\end{enumerate}
	\end{lemma}
	\begin{proof}
 Parts (i) and (ii) follow directly from \cite[Example]{TzbMATH06233336}.

For part (iii), Lemma \ref{huppert} implies that
		$$\rho(L_2(q))={\rho(C_{\frac{q-1}{2}})}^{\frac{q(q+1)}{2}} {\rho(C_{\frac{q+1}{2}})}^{\frac{q(q-1)}{2}}  p^{(q+1)(q-1)}.$$
		Consequently,
		$$\Omega_{\rho}(L_2(q))={\frac{q(q+1)}{2}}{\Omega_{\rho}(C_{\frac{q-1}{2}})}+ {\frac{q(q-1)}{2}} {\Omega_{\rho}(C_{\frac{q+1}{2}})}+ {(q+1)(q-1)}.$$
	\end{proof}

	\begin{lemma}\label{prime}
		Let $G$  be a finite group. Then $|G| \leq 1+\Omega_{\rho}(G)$, with equality holds if and only if 
		every non-trivial element of \( G \) has prime order. 
	\end{lemma}
\begin{proof}
 The proof is straightforward.
\end{proof}
	\begin{theorem}
	Let $G$ be a finite group of order $n$. Then $\Omega_{\rho}(G) \leq \Omega_{\rho}(C_n)$
	with equality if and only if $G$ is cyclic.
\end{theorem}
\begin{proof}
	By Lemma \ref{amiri}, there exists a bijection $ f $ from $ G $ onto $ C_n $
	such that $ o(x) | o(f(x)) $ for all $ x \in G $. Therefore
	\begin{align*}
		\forall x \in G; o(x) | o(f(x)) & \Rightarrow \prod_{x \in G} o(x) \mid \prod_{x \in G} o(f(x)) \\
		& \Rightarrow \rho(G) | \rho(C_n) \Rightarrow \Omega_{\rho}(G) \leq \Omega_{\rho}(C_n).
	\end{align*}
By Lemma \ref{cyclic}, equality holds if and only if $ G=C_n $.
\end{proof}

\begin{lemma} \label{directt}
	Let  $A$ and $B$ be finite groups. Then 
	   \begin{enumerate}
		\item[(1)] \( \rho(A \times B) \mid \rho(A)^{|B|} \rho(B)^{|A|} \). \\
		Moreover, \( \rho(A \times B) = \rho(A)^{|B|} \rho(B)^{|A|} \) if and only if \( \gcd(|A|, |B|) = 1 \).
		
		\item[(2)] \( \Omega_{\rho}(A \times B) \leq \Omega_{\rho}(A) \cdot |B| + \Omega_{\rho}(B) \cdot |A| \). \\
		Equality \( \Omega_{\rho}(A \times B) = \Omega_{\rho}(A) \cdot |B| + \Omega_{\rho}(B) \cdot |A| \) holds if and only if \( \gcd(|A|, |B|) = 1 \).
	\end{enumerate}
\end{lemma}
\begin{proof}
(1)
For any \( (a, b) \in A \times B \), we have
$o((a, b)) = \operatorname{lcm}(o(a), o(b)) \mid o(a) o(b)$.
Thus,
$$
	\rho(A \times B) = \prod_{(a, b) \in A \times B}  o((a, b)) 
 \mid \prod_{a \in A} \prod_{b \in B} o(a) o(b)  =\rho(A)^{|B|} \rho(B)^{|A|}
$$

We know that, $\operatorname{gcd}(|A|,|B|)=1$ if and only if $\operatorname{gcd}(o(a), o(b))=1$ for every $a \in A$ and $b \in B$, which is equivalent to
 $o((a, b))=o(a) o(b)$ for every $a \in A$ and $b \in B$. Therefore 
 $\rho(A \times B) = \rho(A)^{|B|} \rho(B)^{|A|}$ if and only if $\operatorname{gcd}(|A|,|B|)=1$. \\
(2) From part (1), we deduce
 \[\Omega_{\rho}(A\times B) \leqslant
 \Omega_{\rho}({\rho(A)}^{|B|}{\rho(B)}^{|A|}) = \Omega_{\rho}({\rho(A)}^{|B|}) + \Omega_{\rho}({\rho(B)}^{|A|}) = 
 \Omega_{\rho}(A) \cdot |B| + \Omega_{\rho}(B) \cdot |A|.
 \]
 Equality holds precisely when \( \gcd(|A|, |B|) = 1 \), as established in part (1),
 as required.
\end{proof}

\begin{lemma}\label{nim}
	Let $G$ be a finite group satisfying $G=P\rtimes F$, where $P$ is a
	cyclic $p$-group for some prime  $p$, $|F| > 1$ and $(p, |F|) = 1$. Then
	\[\Omega_{\rho}(G)={|C_F(P)|}\cdot {\Omega_{\rho}(P)} + {|P|} \cdot {\Omega_{\rho}(F)}.  \]	
\end{lemma}
\begin{proof} Using Lemma \ref{mid1}, we have
	\begin{align*}	
			\Omega_{\rho}(G)&=\Omega_{\rho}({\rho(P)}^{|C_F(P)|}{\rho(F)}^{|P|}) =
			\Omega_{\rho}({\rho(P)}^{|C_F(P)|})+\Omega_{\rho}({\rho(F)}^{|P|}) \\ &=
		{|C_F(P)|}\cdot {\Omega_{\rho}(P)} + {|P|} \cdot {\Omega_{\rho}(F)}.
	\end{align*} 
	as wanted.
\end{proof}
\begin{lemma}\label{sylowproduct}	
	Let $P$ be a cyclic normal Sylow $p$-subgroup of a finite group $G$. Then
	\[\Omega_{\rho}(G) \leq	|P| \cdot \Omega_{\rho}(G/P) + |G/P| \cdot  \Omega_{\rho}(P)\]
	with equality if and only if $P$ is central in $G$.
\end{lemma}
\begin{proof}
	Using Lemma \ref{1sylowproductt}, we have:
		\[\rho(G) \mid	\rho(G/P)^{|P|}\rho(P)^{|G/P|}\]
		with equality if and only if $P$ is central in $G$.
		Therefore
			\begin{align*}	
		\Omega_{\rho}(G) \leq	\Omega_{\rho}(\rho(G/P)^{|P|}\rho(P)^{|G/P|})& = \Omega_{\rho}(\rho(G/P)^{|P|}) + \Omega_{\rho}(\rho(P)^{|G/P|})  \\ & = |P| \cdot \Omega_{\rho}(G/P) + |G/P| \cdot  \Omega_{\rho}(P).
			\end{align*} 
		with equality if and only if $P$ is central in $G$.
\end{proof}
\begin{corollary}\label{moshtaghkasri22}
	Let $P$ be a cyclic normal Sylow $p$-subgroup of a finite group $G$. Then
	\[\Omega_{\rho}(G/P) \geq \dfrac{\Omega_{\rho}(G)\cdot|P|-\Omega_{\rho}(P)\cdot |G|}{{|P|}^2}.\]
\end{corollary}
\begin{proof}
	Using Lemma \ref{sylowproduct}, we have
	\begin{align*}
		&	\Omega_{\rho}(G) \leq	|P| \cdot \Omega_{\rho}(G/P) + |G/P| \cdot  \Omega_{\rho}(P)
		\Rightarrow  |P| \cdot \Omega_{\rho}(G/P) \geq \Omega_{\rho}(G) - |G/P| \cdot  \Omega_{\rho}(P) \\
		& \Rightarrow    \Omega_{\rho}(G/P) \geq \dfrac{\Omega_{\rho}(G) - |G/P| \cdot  \Omega_{\rho}(P)}{|P|} \Rightarrow    \Omega_{\rho}(G/P) \geq \dfrac{\Omega_{\rho}(G) \cdot |P| -  \Omega_{\rho}(P) \cdot  |G|  }{|P|^2}.
	\end{align*}	
\end{proof}
\begin{corollary}\label{moshtaghkasri}
Let $P$ be a cyclic normal Sylow $p$-subgroup of a finite group $G$ with $P$  central in $G$. Then
\[\Omega_{\rho}(G/P)=\dfrac{\Omega_{\rho}(G)\cdot|P|-\Omega_{\rho}(P)\cdot |G|}{{|P|}^2}.\]
\end{corollary}

\begin{theorem}\label{A4}
	Let $G$ be a finite  group of order $12$. Then
	   \begin{enumerate}
		\item[(i)] \( \Omega_{\rho}(A_4) \leq \Omega_{\rho}(G) \) for all such \( G \),
		\item[(ii)] \( \Omega_{\rho}(G) = 11 \) if and only if \( G \cong A_4 \).
	\end{enumerate}
\end{theorem}
\begin{proof}
	Using GAP, we have
   \begin{align*}
&	\rho(C_3 \rtimes C_4) = 2^{15}3^4,  & \Omega_{\rho}(C_3 \rtimes C_4) &= 19, \\
&	\rho(C_{12}) = 2^{15}3^8,           & \Omega_{\rho}(C_{12}) &= 23, \\
&	\rho(A_4) = 2^33^8,                 & \Omega_{\rho}(A_4) &= 11, \\
&	\rho(D_{12}) = 2^93^4,              & \Omega_{\rho}(D_{12}) &= 13, \\
&	\rho(C_6 \times C_2) = 2^93^8,      & \Omega_{\rho}(C_6 \times C_2) &= 17.
\end{align*}
We get the result.
\end{proof}
\begin{theorem}\label{A5}
	Let $G$ be a finite  group of order $60$. Then $\Omega_{\rho}(A_5) \leqslant \Omega_{\rho}(G)$ and
	 $\Omega_{\rho}(G)=59$   if and only if $G  \cong A_5$.
\end{theorem}
\begin{proof}
	It is enough to show that $\Omega_{\rho}(A_5) < \Omega_{\rho}(G)$, when $G$ is
	solvable. 
We know $|G|=2^2\cdot 3 \cdot 5$ and $\Omega_{\rho}(A_5)=59$. Therefore $\rho(G)=2^{\alpha_2}3^{\alpha_3}5^{\alpha_5}$ and $\Omega_{\rho}(G)={\alpha_2}+{\alpha_3}+{\alpha_5}$. Using Lemma \ref{mohem}, we have
\begin{align}\label{1}
	15 \mid \alpha_2; \qquad 20 \mid \alpha_3; \qquad 12 \mid \alpha_5.
\end{align}
If any of  $\alpha_2\neq 15$,  $\alpha_3 \neq 15$ or $\alpha_5 \notin \{15, 24\}$ holds, then by (\ref{1}) implies  $59 < \Omega_{\rho}(G)$ and we get the result.
Thus, we may assume $\rho(G)=2^{15}3^{20}5^{\alpha_5}$, where $\alpha_3 \in \{12, 24\}$.
By Lemma \ref{hall}, we get that $G$ has a unique Sylow $5$-subgroup, say $P$. Since $G$ is solvable, $G \cong P \rtimes F$, where $F$ is a group of order $12$.
Applying Lemma \ref{mid1},
	\[\rho(G) = \rho(P)^{|C_{F}(P)|} \rho(F)^{|P|}\Rightarrow 2^{15}3^{20}5^{\alpha_5}=(5^4)^{|C_{F}(P)|} \rho(F)^{5}.\]
	Therefore $\rho(F)=2^{3}3^{4}$ and consequently $\Omega_{\rho}(F)=7$.
	However, this contradicts Lemma \ref{A4}, which states that $\Omega_{\rho}(F) \geq 11$ for groups $F$ of order 12.
\end{proof}
\begin{theorem}
	Let $G$ be a finite  group of order $660$. Then $\Omega_{\rho}(L_2(11)) \leqslant \Omega_{\rho}(G)$ and
	$\Omega_{\rho}(G)=769$   if and only if $G  \cong L_2(11)$.
\end{theorem}
\begin{proof}
We first show that for solvable groups $G$,  $\Omega_{\rho}(L_2(11)) < \Omega_{\rho}(G)$.
	We know $|G|=2^2\cdot 3 \cdot 5 \cdot 11$ and $\Omega_{\rho}(L_2(11))=769$. Let  $\rho(G)=2^{\alpha_2}3^{\alpha_3}5^{\alpha_5}5^{\alpha_{11}}$ with $\Omega_{\rho}(G)={\alpha_2}+{\alpha_3}+{\alpha_5}+{\alpha_{11}}$. Using Lemma \ref{mohem}, we have
	\begin{align}\label{22}
	165 \mid \alpha_2; \qquad  220 \mid \alpha_3; \qquad   132 \mid  \alpha_5; \qquad  60 \mid  \alpha_{11}.
	\end{align}

	If $\alpha_2\neq 165$, then by (\ref{22}) and Lemma \ref{2-part}, $\alpha_2\geqslant495$. Thus
	$\Omega_{\rho}(G)={\alpha_2}+{\alpha_3}+{\alpha_5}+{\alpha_{11}}\geqslant 495+220+132+60>769,$
	 and so we get the result. Therefore, we assume that $\alpha_2 = 165$.
	 
	If  $\alpha_3 \neq 220$ or $\alpha_5 \notin \{132, 264\}$, then by (\ref{22}),  $769 < \Omega_{\rho}(G)$ and we get the result.
	Therefore $\rho(G)=2^{165}3^{220}5^{\alpha_5}11^{\alpha_{11}}$, where $\alpha_5 \in \{132, 264\}$.
	By Lemma \ref{hall}, we get that $G$ has a unique Sylow $11$-subgroup, say $P$. Since $G$ is solvable, $G \cong P \rtimes F$, where $F$ is a group of order $60$.
	Using Lemma \ref{mid1},
	\[\rho(G) = \rho(P)^{|C_{F}(P)|} \rho(F)^{|P|}\Rightarrow 2^{165}3^{220}5^{\alpha_5}11^{\alpha_{11}}=(11^{10})^{|C_{F}(P)|} \rho(F)^{11}.\]
	Therefore $\rho(F)=2^{15}3^{20}5^{t_5}$, where $t_5={\alpha_5}/11 \in \{12,24\}$ and so $\Omega_{\rho}(F)\leqslant 59$.
	By Theorem \ref{A5}, we have $F \cong A_5$. Thus $F$ is not solvable and so $G$ is not solvable, which is a contradiction.
	Therefore $G$ is non-solvable.

 Since $G$ is non-solvable, Lemma \ref{seri} implies the existence of a normal series:
 \[ 1 \unlhd H \unlhd K \unlhd G \]
 where the quotient $K/H$ is isomorphic to either $A_5$ or $L_2(11)$, and $|G/K|$ divides $|\operatorname{Out}(K/H)|$.
 
 \noindent\textbf{Case 1:} $K/H \cong A_5$.
 In this case, we have $|H| = 11$ and $G = K$, making $G$ an extension of $C_{11}$ by $A_5$. We know that $G / C_G(H) \hookrightarrow \operatorname{Aut}(H)$ and 
 $(G / H) /\left(C_G(H) / H\right) \cong G / C_G(H)$.
  So $G$ is a central extension of $H$ by $A_5$.
 Since the Schur multiplier of $A_5$ is $C_2$, we get that $G \cong C_{11} \times A_5$.
 Applying Lemma \ref{directt}, we compute:
 \[
 \Omega_{\rho}(C_{11} \times A_5) = \Omega_{\rho}(C_{11}) \cdot |A_5| + \Omega_{\rho}(A_5) \cdot |C_{11}| = 10 \cdot 60 + 59 \cdot 11 = 1249
 \]
 This establishes that $\Omega_{\rho}(C_{11} \times A_5) > \Omega_{\rho}(L_2(11))$.
 
 \noindent\textbf{Case 2:} $K/H \cong L_2(11)$.
  In this case, we immediately obtain 
 $ G \cong L_2(11) $.
 
 The proof is now complete.
\end{proof}
\begin{theorem}
	Let $G$ be a finite  group of order $1092$. Then $\Omega_{\rho}(L_2(13)) \leqslant \Omega_{\rho}(G)$ and
	$\Omega_{\rho}(G)=1273$   if and only if $G  \cong L_2(13)$.
\end{theorem}
\begin{proof}
	We first prove that for solvable groups $G$ of order $2^2 \cdot 3 \cdot 7 \cdot 13$,
	\[ \Omega_{\rho}(L_2(13)) < \Omega_{\rho}(G). \]
We know $|G|=2^2\cdot 3 \cdot 7 \cdot 13$ and $\Omega_{\rho}(L_2(13))=1273$. Let $\rho(G)=2^{\alpha_2}3^{\alpha_3}7^{\alpha_7}13^{\alpha_{13}}$ with $\Omega_{\rho}(G)={\alpha_2}+{\alpha_3}+{\alpha_7}+{\alpha_{13}}$. Using Lemma \ref{mohem}, we have
\begin{align}\label{333}
	273 \mid \alpha_2; \qquad  364 \mid \alpha_3; \qquad   156 \mid  \alpha_7; \qquad  84 \mid  \alpha_{13}.
\end{align}
If $\alpha_2\neq 273$, then by (\ref{333}) and Lemma \ref{2-part}, $\alpha_2\geqslant 819$. Thus
$\Omega_{\rho}(G)={\alpha_2}+{\alpha_3}+{\alpha_5}+{\alpha_{11}}\geqslant 813+364+156+84>1273,$
and so we get the result. Therefore, we assume that $\alpha_2 = 273$.\\
If  $\alpha_3 \notin \{364, 728\}$, then by (\ref{333}),  $1273 < \Omega_{\rho}(G)$ and we get the result.\\
Therefore $\rho(G)=2^{273}3^{\alpha_3}7^{\alpha_7}13^{\alpha_{13}}$, where $\alpha_3 \in \{364, 728\}$.
By Lemma \ref{hall}, we get that $G$ has a unique Sylow $13$-subgroup, say $P$. Since $G$ is solvable, $G \cong P \rtimes H$, where $H$ is a group of order $84$.
Using Lemma \ref{mid1},
\[\rho(G) = \rho(P)^{|C_{H}(P)|} \rho(H)^{|P|}\Rightarrow 2^{273}3^{\alpha_3}7^{\alpha_7}13^{\alpha_{13}}=(13^{12})^{|C_{H}(P)|} \rho(H)^{13}.\]
Therefore $|H|=2^2\cdot 3 \cdot 7$ and $\rho(H)=2^{21}3^{t_3}7^{t_7}$, where $t_3={\alpha_3}/13 \in \{28,56\}$ and $t_7={\alpha_7}/13$.

By Lemma \ref{hall}, we get that $H$ has a unique Sylow $7$-subgroup, say $Q$. Since $H$ is solvable, $H \cong Q \rtimes F$, where $F$ is a group of order $12$.
Using Lemma \ref{mid1},
\[\rho(H) = \rho(Q)^{|C_{F}(Q)|} \rho(F)^{|Q|}\Rightarrow 2^{21}3^{t_3}7^{t_7}=(7^{6})^{|C_{F}(Q)|} \rho(F)^{7}.\]
Therefore $|F|=2^2\cdot 3$ and $\rho(F)=2^{3}3^{r_3}$, where $r_3={t_3}/7 \in \{4, 8\}$. 
Using Theorem \ref{A4}, we have $\Omega_{\rho}(F) \geqslant12 $. Thus  $r_3=8$ and so $t_3=56$ and ${\alpha_3}=728$.
Therefore $\rho(G)=2^{273}3^{720}7^{\alpha_7}13^{\alpha_{13}}$.
If $\alpha_7 \neq 156$ or $\alpha_{13}\neq 84$, then by (\ref{333}) $\alpha_7 > 312$ or $\alpha_{13}> 168$. Thus $\Omega_{\rho}(F)> 1273$ and we get the result.
Therefore $\rho(G)=2^{273}3^{728}7^{156}13^{84}$.

By Lemma \ref{hall}, we get that $G$ has a unique Sylow $7$-subgroup, say $R$. Since $G$ is solvable, $G \cong R \rtimes T$, where $T$ is a group of order $156$.
Using Lemma \ref{mid1},
\[\rho(G) = \rho(R)^{|C_{T}(R)|} \rho(T)^{|R|}\Rightarrow 2^{21}3^{t_3}7^{t_7}=(7^{6})^{|C_{T}(R)|} \rho(T)^{7}.\]
Therefore $|T|=2^2\cdot 3 \cdot 13$ and $\rho(T)=2^{39}3^{104}13^{12}$. Thus $|T|=1+\Omega_{\rho}(T)$.
Using Lemma  \ref{prime}, $ T $ is a group with
non-trivial elements of prime orders. 
By \cite[Theorem]{prime}, $T$ is a $p$-group of exponent $p$ or $T$ is a Frobenius  group of order $p^nq$ or $T \cong A_5$, which is a contradiction.

Therefore $G$ is non-solvable, and so $G \cong L_2(13)$. This completes the proof.
\end{proof}
We end with the following questions:\\
\textbf{Question 1}. What information about a group $G$ can be obtained from
 $\Omega_{\rho}(G)$?\\
\textbf{Question 2}. Which groups $G$ can be uniquely determined by 
$\Omega_{\rho}(G)$?

\end{document}